\documentclass[12pt]{amsart}
\usepackage{amsmath,amssymb,amsbsy,amsfonts,amsthm,latexsym,amsopn,
amstext,amsxtra,euscript,amscd}

\usepackage{color}

\begin{document}

\newtheorem{thm}{Theorem}
\newtheorem{lem}[thm]{Lemma}
\newtheorem{claim}[thm]{Claim}
\newtheorem{cor}[thm]{Corollary}
\newtheorem{prop}[thm]{Proposition}
\newtheorem{definition}{Definition}
\newtheorem{question}[thm]{Open Question}
\newtheorem{conj}[thm]{Conjecture}
\newtheorem{prob}{Problem}

\def\squareforqed{\hbox{\rlap{$\sqcap$}$\sqcup$}}
\def\qed{\ifmmode\squareforqed\else{\unskip\nobreak\hfil
\penalty50\hskip1em\null\nobreak\hfil\squareforqed
\parfillskip=0pt\finalhyphendemerits=0\endgraf}\fi}

\def\cA{{\mathcal A}}
\def\cB{{\mathcal B}}
\def\cC{{\mathcal C}}
\def\cD{{\mathcal D}}
\def\cE{{\mathcal E}}
\def\cF{{\mathcal F}}
\def\cG{{\mathcal G}}
\def\cH{{\mathcal H}}
\def\cI{{\mathcal I}}
\def\cJ{{\mathcal J}}
\def\cK{{\mathcal K}}
\def\cL{{\mathcal L}}
\def\cM{{\mathcal M}}
\def\cN{{\mathcal N}}
\def\cO{{\mathcal O}}
\def\cP{{\mathcal P}}
\def\cQ{{\mathcal Q}}
\def\cR{{\mathcal R}}
\def\cS{{\mathcal S}}
\def\cT{{\mathcal T}}
\def\cU{{\mathcal U}}
\def\cV{{\mathcal V}}
\def\cW{{\mathcal W}}
\def\cX{{\mathcal X}}
\def\cY{{\mathcal Y}}
\def\cZ{{\mathcal Z}}

\def\ord{{\mathrm{ord}\,}}
\def\ind{{\mathrm{ind}\,}}
\def \F {{\mathbb F}}
\def \L {{\mathbb L}}
\def \K {{\mathbb K}}
\def \Q {{\mathbb Q}}
\def \R {{\mathbb R}}
\def \Z {{\mathbb Z}}

\def \fB{\mathfrak B}

\def\\{\cr}
\def\({\left(}
\def\){\right)}
\def\fl#1{\left\lfloor#1\right\rfloor}
\def\rf#1{\left\lceil#1\right\rceil}

\newcommand{\pfrac}[2]{{\left(\frac{#1}{#2}\right)}}

\def\sssum{\mathop{\sum\!\sum\!\sum}}
\def\ssum{\mathop{\sum\, \sum}}

\def \Prob{{\mathrm {}}}
\def\e{\mathbf{e}}
\def\ep{{\mathbf{\,e}}_p}
\def\epp{{\mathbf{\,e}}_{p^2}}
\def\em{{\mathbf{\,e}}_m}

\def\Res{\mathrm{Res}}

\def\tz{\widetilde z}
\def\tX{\widetilde X}
\def\tx{\widetilde x}
\def\tu{\widetilde u}
\def\tv{\widetilde v}
\def\tw{\widetilde w}

\renewcommand{\vec}[1]{\mathbf{#1}}

\def\va{\vec{a}}
\def\vb{\vec{b}}
\def\Ha{H_{\vec{a}}}
\def\Hb{H_{\vec{b}}}

\def \li {\mathrm {li}\,}

\def\mand{\qquad\mbox{and}\qquad}

\newcommand{\comm}[1]{\marginpar{%
\vskip-\baselineskip 
\raggedright\footnotesize
\itshape\hrule\smallskip#1\par\smallskip\hrule}}


\title[Points on Curves in Small Boxes and Applications]
{Points on Curves in Small Boxes and Applications}

\author[M.-C. Chang {\it etal.}]
{Mei-Chu Chang}
\address{Department of Mathematics, University of California,
Riverside,  CA 92521, USA}
\email{mcc@math.ucr.edu}

\author[ ]
{Javier Cilleruelo}
\address{Instituto de Ciencias Matem\'{a}ticas (CSIC-UAM-UC3M-UCM) and Departamento de Matem\'aticas, Universidad Aut\'onoma de Madrid,
28049, Madrid, Espa\~na}
\email{franciscojavier.cilleruelo@uam.es}

\author[ ]
{Moubariz~Z.~Garaev}
\address{Centro de Ciencias Matem\'{a}ticas, Universidad Nacional Aut\'onoma de M\'{e}xico,
C.P. 58089, Morelia, Michoac\'{a}n, M\'{e}xico}
\email{garaev@matmor.unam.mx}
\author[ ]
{Jos\'e Hern\'andez}
\address{Centro de Ciencias  Matem\'{a}ticas, Universidad Nacional Aut\'onoma de M\'{e}xico,
C.P. 58089, Morelia, Michoac\'{a}n, M\'{e}xico}
\email{stgo@matmor.unam.mx}
\author[ ]
{Igor E. Shparlinski}
\address{Department of Computing, Macquarie University, Sydney, NSW 2109, Australia}
\email{igor.shparlinski@mq.edu.au}
\author[ ]
{Ana Zumalac\'{a}rregui}
\address{Instituto de Ciencias Matem\'{a}ticas (CSIC-UAM-UC3M-UCM) and Departamento de Matem\'aticas, Universidad Aut\'onoma de Madrid,
28049, Madrid, Espa\~na} \email{ana.zumalacarregui@uam.es}

\date{}

\begin{abstract}  We introduce
several new methods to obtain upper bounds on the
number of solutions of the congruences
$$
f(x) \equiv y \pmod p \mand f(x) \equiv y^2 \pmod p,
$$
with a prime $p$ and a  polynomial $f$,
where $(x,y)$ belongs to an
arbitrary square with side length $M$. We use these results and methods to
derive non-trivial
upper bounds for the number of
hyperelliptic curves
$$Y^2=X^{2g+1} + a_{2g-1}X^{2g-1} + \ldots + a_1X+a_0$$
over  the finite field $\F_p$ of $p$ elements, with coefficients in a $2g$-dimensional cube
$$
(a_0,\ldots, a_{2g-1})\in [R_0+1,R_0+M]\times \ldots \times
[R_{2g-1}+1,R_{2g-1}+M]
$$
that are isomorphic to a given curve and give an almost sharp
lower bound on the number of non-isomorphic hyperelliptic curves
with coefficients in that cube. Furthermore, we study the size of
the smallest box that contain a partial trajectory of a
polynomial dynamical system over $\F_p$.
\end{abstract}

\maketitle


\section{Introduction}

\subsection{Motivation}

Studying the distribution of integer and rational points on curves, and more general on
algebraic varieties that belong to a given box is a classical topic in analytic
number theory. For the case of plane curves with integer coefficients,
essentially the best
possible results are due to Bombieri and Pila~\cite{BoPi,Pil1,Pil2}. Furthermore,
recently a remarkable progress has been made in the case
of hypersurfaces  and varieties over the rationals,
see the surveys~\cite{Brow,HB2,Tsch}
as well as the original works~\cite{Marm1,Marm2,SalWool}.

Significantly less is
known about the distribution of  points in  boxes on curves and
varieties in finite fields. For reasonably large boxes, bounds of
exponential sums, that in turn are based on deep methods of algebraic
geometry, lead to asymptotic formulas for the number of such points,
see~\cite{Fouv,FoKa,Luo}. Certainly when the size of the box is decreasing
then beyond a certain threshold no asymptotic formula is possible (in fact
the expected number of points can be less than 1). In particular, for such
a small box only one can expect  to derive upper bounds on the
number of points on curves that hit it. This question has recently been
introduced in~\cite{CGOS} where a series of general results has been
obtained (we also mention the work~\cite{ChShp,CillGar,Zum} where this
question has been studied for some very special curves). Besides of being interesting in their own right,
such results have several applications, for example, to
polynomial dynamical systems and to so-called ``visible''
points on curves  over finite fields, see~\cite{CGOS}.

Here we make more advances in this direction, consider more general
curves and improve several previous results. However, we consider that
then main feature of the paper is a variety of methods we introduce
which we believe can be used for a number of other questions.

We also give two further applications of our results and methods.
First of all, we study the distribution of isomorphism classes of hyperelliptic curves
of genus $g\ge 1$ in
some families of curves associated with polynomials with coefficients in
a small box. In the case of elliptic curves this question has been
studied in~\cite{CSZ}. Here we improve some of the results of~\cite{CSZ}
and also use new methods to study the case of $g\ge 2$. Surprisingly enough,
in the case of the genus $g\ge 2$ we obtain estimates and use methods that do
not apply to elliptic curves (that is, to $g=1$).

Second, we consider polynomial dynamical systems and study
for how long a particular trajectory of such a system
can be ``locked'' in a given box. In particular, we extend and
improve several results of~\cite{Chang,CGOS,GuShp}.

\subsection{Basic definitions and problem formulation}

For a prime $p$, let  $\F_p$ denote the finite field of $p$ elements, which we assume to be represented by the set $\{0,1,\ldots, p-1\}$. Given a polynomial $f \in \F_p[X]$ of degree $m\ge 3$, and a positive integer $M<p$, we define by $I_{f}(M;R,S)$ the number of solutions to the congruence
\begin{equation}
\label{eq:cong y2=f(x)}
y^2\equiv f(x) \pmod p,
\end{equation}
with
\begin{equation}
\label{eq:square}
(x,y)\in [R+1, R+M]\times [S+1, S+M].
\end{equation}
If the polynomial $y^2-f(x)$ is absolutely irreducible, it is known from the Weil bounds that
\begin{equation}
\label{eq:Asymp I(M;R,S)}
I_{f}(M;R,S)=\frac{M^2}{p}+O(p^{1/2}(\log p)^2),
\end{equation}
where the implied constant depends only on $m$,
see~\cite{VajZah,Zheng}.
It is clear that the main term is dominated by the error term for
$M\le p^{3/4}\log p$, and for $M\le p^{1/2}(\log p)^2$ the result
becomes weaker than the trivial upper bound $I_{f}(M;R,S)\le 2M$. Here we use a
different approach and give nontrivial estimate of $I_{f}(M;R,S)$ for $M <
p^{1/4-\varepsilon}$ when $m=3$, and for $M<p^{1/3-\varepsilon}$ when
$m\ge 4$. In particular, in the case $m=3$ our result improves on
the range of $M$ the bound obtained in~\cite{CSZ}.
We note that  nontrivial bounds on the number of
solutions $(x,y)$ to the congruence
$$
y\equiv f(x)\pmod p,
$$
satisfying~\eqref{eq:square}, have been obtained in~\cite{CGOS} for
any $M<p$. We also mention that nontrivial bounds on the number of
solutions $(x,y)$ to the congruences
$$
xy \equiv a \pmod p,
$$
and
$$
y\equiv \vartheta^x  \pmod p,
$$
satisfying~\eqref{eq:square}, have been given in~\cite{ChShp} with
further improvements in~\cite{CillGar}. Similar results
for  the congruence
$$
Q(x,y)\equiv 0 \pmod p,
$$
where $Q(x,y)$ is an absolutely irreducible quadratic form with
a nonzero discriminant, can be found in~\cite{Zum}.

A special case of the equation~\eqref{eq:cong y2=f(x)} are
hyperelliptic curves over $\F_p$.  The problem of concentration of
points on hyperelliptic curves and polynomial maps is connected with
some problems on isomorphisms that preserve hyperelliptic curves. Let
$g$ be a fixed positive integer constant. We always assume that $p$ is
large enough so, in particular, we have $\gcd(p,2(2g+1))=1$.
Any hyperelliptic curve
can be given by a non-singular {\it Weierstrass equation:\/}
$$
\Ha: \quad Y^2 = X^{2g+1} + a_{2g-1}X^{2g-1} + \ldots +
a_1X+a_0,
$$
where $\va=(a_0,\ldots ,a_{2g-1}) \in \F_p^{2g}$
(the non-singularity condition is equivalent to non-vanishing of
the discriminant of
$X^{2g+1} + a_{2g-1}X^{2g-1} + \ldots +a_1X+a_0$), we refer
to~\cite{ACDFLNV}
for a background on hyperelliptic curves and their applications.

It follows from a more general result of Lockhart~\cite[Proposition 1.2]{Lock}
that isomorphisms that preserve hyperelliptic curves given by
Weierstrass equations are all of the form
$(x,y)\rightarrow (\alpha^2 x,\alpha^{2g+1}y)$
for some $\alpha \in \F_p^*$, see also~\cite[Section~3]{KU}.
Thus $\Ha$ is isomorphic to $\Hb$, which we
denote as $\Ha \sim \Hb$,
if there exists  $\alpha \in \F_p^*$ such that
\begin{equation}
\label{eq:isom}
a_i\equiv \alpha^{4g+2-2i}b_i\pmod p, \qquad
i=0,\ldots , 2g-1.
\end{equation}

It is known (see~\cite{KU,Nart}) that the number of non
isomorphic hyperelliptic curves of genus $g$ over $\F_p$ is
$2p^{2g-1}+O(gp^{2g-2})$. We address here the problem of estimating
from below, the number of non-isomorphic hyperelliptic curves of
genus $g$ over $\F_p$, $\Ha$, when $\va=(a_0,\ldots
,a_{2g-1})$ belongs to a small $2g$-dimensional cube
\begin{equation}
\label{eq:cube B}
\fB=[R_0+1,R_0+M]\times \ldots \times [R_{2g-1}+1,R_{2g-1}+M]
\end{equation}
with some integers $R_j$, $M$ satisfying $0 \le R_j < R_j + M < p$,
$j =0, \ldots, 2g-1$.

In particular, we note that all components of
a vector $\va\in \fB$ are  non-zero
modulo $p$. Our methods below work without this restriction
as well, however they somewhat lose their efficiency.

We also give an upper bound for the number
$$N(H;\fB)=\# \{\va=(a_0,\ldots ,a_{2g-1})\in \fB:\ \Ha\sim
H\}$$
of hyperelliptic curves
$\Ha$ with  $\va\in \fB$ that are isomorphic to a given curve $H$.

In particular, our estimates extend and improve some of the
results of~\cite{CSZ} where this problem has been
investigated for elliptic
curves (that is, for $g=1$).

First we observe that for large cubes one easily derives
from the Weil bound
 (see~\cite[Chapter~11]{IwKow}) an asymptotic formula
$$
N(H;\fB) = \frac{M^{2g}}{p^{2g-1}} + O(p^{1/2} (\log p)^{2g})
$$
(see also the proof of~\cite[Theorem~21.4]{IwKow}).
So we have an asymptotic formula for $N(H;\fB)$ as long as
$M \ge p^{1-1/(4g) + \varepsilon}$ for any fixed $\varepsilon > 0$.

However here we are mostly interested in small values of $M$.

We note that we always have the trivial upper bound
$$
N(H;\fB)\le 2M.
$$
To see this, let $H=\Hb$, $\vb=(b_0,\ldots,b_{2g-1})\in \F_p^{2g}$,
be given by a Weierstrass equation. We observe  that
if $\Ha\sim H$ and   $H=\Hb$, where $\vb=(b_0,\ldots ,b_{2g-1}) \in \F_p^{2g}$ then
$a_{2g-1}$ can take at most $M$ values in $\F_p^*$, and each $a_{2g-1}$
determines two possible values for $\alpha^2$ in~\eqref{eq:isom}.

It is also useful to remark that one can not expect to get a
general bound stronger than
$$
N(H;\fB) = O(M^{1/(2g+1)}).
$$
To see this we consider the set
$\cQ$ of quadratic residues modulo $p$ in the interval $[1, M^{1/(2g+1)}]$.
It is well-known that for almost all primes $p$ (that is, for all
except a set of relative density zero) we have
$$
\# \cQ \sim 0.5 M^{1/(2g+1)}.
$$
For example, this follows  from  a bound of
Heath-Brown~\cite[Theorem~1]{HB1} on average values
of sums of real characters.

Consider now the set
$$
\cA=\{\alpha \in \F_p~:~ \alpha^2 \in \cQ\},
$$ the curve $H:
Y^2=X^{2g+1}+X^{2g-1}+X^{2g-2}+\ldots +X+1$ and the $2g$-dimensional cube
$\fB=[1,M]^{2g}$. It is clear that  $(\alpha^4, \alpha^6,\ldots ,
\alpha^{4g+2})\in \fB$ for all $\alpha \in \cA$. On the other hand
$\# \cA = 2 \# \cQ \sim  M^{1/(2g+1)}$.

\subsection{General notation}

Throughout the paper, any implied constants in the symbols $O$, $\ll$
and $\gg$ may occasionally depend, where obvious, on the
degree of polynomial $f\in \F_p[X]$, on the genus $g$ and the real positive
parameters  $\varepsilon$ and $\delta$, and are absolute otherwise. We recall
that the notations $U = O(V)$,  $U \ll V$ and  $V \gg U$  are
all equivalent to the statement that $|U| \le c V$ holds
with some constant $c> 0$.

The letters, $h,m, n, r, s$ in both upper and lower case,
always denote
integer numbers.

\section{Our Results}

\subsection{Points on curves in small boxes}
We
combine ideas from~\cite{CillGar,CGOS,CSZ} with some  new ideas and derive
the following results.

\begin{thm}
\label{thm:ConcentElliptic}
Uniformly over all polynomials $f\in \F_p[X]$ of degree $\deg f =3$ and $1\le M < p$,
we have
$$
I_{f}(M;R,S)<M^{1/3+o(1)}+\frac{M^{5/3+o(1)}}{p^{1/6}},
$$
as $M \to \infty$.
\end{thm}

One of the implications of Theorem~\ref{thm:ConcentElliptic} is that
for elliptic curves, that is, when the polynomial $f$ in~\eqref{eq:cong y2=f(x)}
if cubic, the bound  $I_{f}(M;R,S) <
M^{1/3+o(1)}$
holds for $M \ll p^{1/8}$, while~\cite[Theorem~6]{CSZ}
guarantees this bound only for $M\ll p^{1/9}$. We
also note that when $\deg f=3$,  our upper bounds for $I_{f}(M;R,S)$ imply the
same bounds for $N(H;\fB)$ in the case of elliptic curves.

Further, when $M<p^{1/4-\varepsilon}$ for some $\varepsilon>0$,
Theorem~\ref{thm:ConcentElliptic} guarantees a nontrivial bound
$I_{f}(M;R,S)\ll M^{1-\delta}$ with some  $\delta>0$ that depends
only on $\varepsilon$,  improving upon the range
$M<p^{1/5-\varepsilon}$ obtained in~\cite{CSZ}. However, using a
different approach, that is based on the ideas of~\cite{BGKS},
we can obtain a nontrivial bound in the range
$M<p^{1/3-\varepsilon}$.

\begin{thm}
\label{thm:ConcentEllipticWiderRange}
Uniformly over all polynomials $f\in \F_p[X]$ of degree $\deg f =3$
and $M\ge 1$, we have
$$
I_{f}(M;R,S)\le  M^{1/3+o(1)}+\(\frac{M^3}{p}\)^{1/16} M^{1+o(1)}.
$$
\end{thm}

The combination of Theorems~\ref{thm:ConcentElliptic}
and~\ref{thm:ConcentEllipticWiderRange} gives the following estimate:

\begin{cor}
\label{cor:ConcentElliptic} Uniformly over all polynomials $f\in
\F_p[X]$ of degree $\deg f =3$ and $1\le M < p$, we have
$$
I_{f}(M;R,S)< M^{1+o(1)}\left\{\begin{array}{llll}
M^{-2/3} ,&\quad\text{if } \ M<p^{1/8},\\
(M^4/p)^{1/6}  ,&\quad\text{if }   \ p^{1/8}\le M< p^{5/23}, \\
(M^3/p)^{1/16}  ,&\quad\text{if } \ p^{5/23}\le M<p^{1/3},
\end{array}
\right.
$$
as $M \to \infty$.
\end{cor}

Our next result shows that when $\deg f\ge 4$ we also have a nontrivial bound for $I_{f}(M;R,S)$ in the range $M<p^{1/3-\varepsilon}$.

To formulate our result, we define $J_{k,m}(H)$ as the number of solutions of the
system of $m$ diophantine equations in $2k$ integral variables $x_1,\ldots,x_{2k}$:
\begin{equation}
\label{eq:System}
\begin{split}
x_1^m+\ldots +x_k^m&=x_{k+1}^m+\ldots x_{2k}^m,\\
&\ldots  \\x_1+\ldots +x_k&=x_{k+1}+\ldots x_{2k},\\
1\le x_1,&\ldots,x_{2k}\le H.
\end{split}
\end{equation}
We also define  $\kappa(m)$ to be the smallest integer $\kappa$ such
that for any integer $k \ge \kappa$ there exists a constant
$C(k,m)$ depending only on $k$ and $m$ and such that
\begin{equation}
\label{eq:bound J}
J_{k,m}(H)\le C(k,m) H^{2k-m(m+1)/2+o(1)},
\end{equation}
as   $H\to \infty$.
Note that by a recent result of Wooley~\cite[Theorem~1.1]{Wool2},
that improves the previous estimate of~\cite{Wool1},  we
have $\kappa(m) \le m^2-1$ for any $m \ge 3$.

\begin{thm}\label{thm:ConcentDegf>3}
Uniformly over all polynomials $f\in \F_p[X]$ of degree $\deg f =m\ge 4$
and $1\le M < p$, we have
$$
I_{f}(M;R,S) \le
M(M^3/p)^{1/2\kappa(m) +o(1)}+M^{1- (m-3)/2\kappa(m)+o(1)},
$$
as $M\to \infty$.
\end{thm}

In particular, for any $\varepsilon>0$, there exists $\delta>0$ that depends only
on $\varepsilon$ and $\deg f$ such that if $M<p^{1/3-\varepsilon}$ and
$\deg f\ge 4$, then  $I_{f}(M;R,S)\ll M^{1-\delta}$.

Next, we turn to estimates on $N(H; \fB)$.
A simple observation shows that in the case of
hyperelliptic curves with $g\ge 2$ the quantity  $N(H; \fB)$ is
closely related to the problem of concentration of points of a
quadratic polynomial map. Then one can apply the general result
of~\cite{CGOS} and get a nontrivial upper bound for $N(H; \fB)$ for
any range of $M$. However, here we use a different approach and we
obtain a better bound. We  prove the following result, which,
besides of its application to bound the quantity $N(H; \fB)$, is of
independent interest.

\begin{thm}
\label{thm:Polynomial Map}
Let $f\in \F_p[X]$ be a polynomial of degree $m\ge 2$. Then for $1\le M < p$ the
number $J_f(M;R,S)$
of solutions to the congruence
$$
f(x) \equiv y\pmod p, \quad (x,y)\in [R+1, R+M]\times [S+1, S+M],
$$
is bounded by
$$
J_f(M;R,S) \ll \frac{M^2}{p} + M^{1-1/2^{m-1}}p^{o(1)}
$$
as $p\to \infty$.
\end{thm}

We remark that in~\cite{CGOS}, the bound
$$
J_f(M;R,S) \ll  M (M/p)^{1/2\kappa(m)+o(1)} + M^{1- (m-1)/2\kappa(m)+o(1)}
$$
has been given which is stronger than Theorem~\ref{thm:Polynomial Map} for
large values of $m$.
Also in~\cite{CGOS} for  $M \le p^{2/(m^2+3)}$,    the bound
\begin{equation}
\label{eq:J small M}
J_f(M;R,S) \ll   M^{1/m+o(1)}
\end{equation}
has been obtained.

\subsection{Isomorphism classes of hyperelliptic curves in thin families}

Using~\eqref{eq:isom},
we derive  from Theorem~\ref{thm:Polynomial Map} and~\eqref{eq:J small M}
the following consequence

\begin{cor}
\label{cor:HyperellipticAnyRange}
For any hyperelliptic curve $H$ of genus $g\ge 2$ over $\F_p$ and
a cube $\fB$ given by~\eqref{eq:cube B} with $1\le M <p$, we have
$$
N(H;\fB) \ll \frac{M^2}{p} + M^{1/2+o(1)}.
$$
\end{cor}

Furthermore, as we have mentioned above,
when $g=1$ the  problem of estimating $N(H;\fB)$ is equivalent to estimating
the concentration
of points on certain curves of degree $3$ (which are singular and thus are
not elliptic curves) and
Theorem~\ref{thm:ConcentElliptic} applies in this case. Using the idea
of the proof of Theorem~\ref{thm:ConcentElliptic}, we  establish
the following result which is valid for any hyperelliptic curve.

\begin{thm}
\label{thm:HyperellipticOptimal}
For any hyperelliptic curve $H$ of genus $g\ge 1$ over $\F_p$, any
 cube $\fB$ given by~\eqref{eq:cube B} with $1\le M < p$  and any odd integer $h\in [3, 2g+1]$,
we have
$$
N(H;\fB)<
\(M^{1/h} +M\(M^4/p\)^{2/h(h+1)}\)M^{o(1)},
$$
as $M \to \infty$.
\end{thm}

We observe that if $M<p^{1/(2g^2+2g+4)}$ then, taking $h=2g+1$ in
Theorem~\ref{thm:HyperellipticOptimal},  we obtain the estimate
$N(H;\fB)\le M^{1/(2g+1)+o(1)}$ which, as we have seen,
is sharp up to the $o(1)$ term.

Let $\cH\(\fB\)$ be a collection of representatives of all
isomorphism classes of hyperelliptic curves $\Ha$, $\va \in \fB$,
where $\fB$ is a $2g$-dimensional cube of side length $M$.
In~\cite{CSZ} the lower bound $\# \cH\(\fB\)\gg
\min\{p,M^{2+o(1)}\}$ has been obtained for elliptic curves (that
is, for $g=1$). We extend this result to $g\ge 2$. Certainly the
upper bounds of our theorems lead to a lower bound on $\#
\cH\(\fB\)$. However, using a different approach we obtain a
near optimal bound for $\# \cH\(\fB\)$.

\begin{thm}\label{thm:main3} For $g\ge 1$ and any
 cube $\fB$ given by~\eqref{eq:cube B} with and $1\le M < p$,  we have
$$\# \cH\(\fB\)\gg \min\{p^{2g-1},M^{2g+o(1)}\},$$
as $M \to \infty$.
Furthermore, if $g\ge 2$ the $o(1)$ term can be removed when
$M>p^{1/(2g)}$.
\end{thm}

\subsection{Diameter of polynomial dynamical systems}

We  notice that results about concentration of points on curves are closely
related to the question about the diameter of partial trajectories
of polynomial dynamical systems. Namely, given a polynomial $f \in \F_p[X]$
and an element
$u_0 \in \F_p$,
we consider the sequence of elements
of $\F_p$ generated by iterations $u_n = f(u_{n-1})$, $n = 0,1, \ldots$.
Clearly the sequence $u_n$ is eventually periodic. In particular,
let $T_{f, u_0}$ be the full trajectory length, that is, the smallest integer
$t$ such that $u_t = u_s$ for some $s < t$.
The study of the diameter
$$
D_{f, u_0}(N) = \max_{0 \le k,m\le N-1} |u_k - u_m|
$$
has been  initiated in~\cite{GuShp} and then continued in~\cite{Chang,CGOS}.
In particular, it follows from~\cite[Theorem~6]{GuShp}
that for any fixed $\varepsilon$, for $T_{f, u_0} \ge N \ge p^{1/2+\varepsilon}$
we have the asymptotically best possible bound
$$
D_{f, u_0}(N) = p^{1+o(1)}
$$
as $p\to \infty$.
For smaller values of $N$ a series of lower bounds
on $D_{f, u_0}(N)$ is given in~\cite{Chang,CGOS}.

One easily derives from Theorem~\ref{thm:Polynomial Map}
the following result, which improves previous results to intermediate
values of $N$ (and is especially effective for small values of $m$).

\begin{cor}
\label{cor:DiamDynSyst}
For any  polynomial $f \in \F_p[X]$ of degree $m\ge 2$ and positive
integer $N\le T_{f, u_0}$,  we have
$$
D_{f, u_0}(N) \gg \min\{N^{1/2} p^{1/2}, N^{1 + 1/(2^{m-1} -1)}
p^{o(1)}\},
$$
as $p \to \infty$.
\end{cor}

On the other hand, we remark that our method and results do  not affect the superpolynomial lower
bounds of~\cite{Chang} that hold for small values of $N$.

\section{Preparations}

\subsection{Uniform distribution and exponential sums}

The following result is well-known and can be found, for example, in~\cite[Chapter~1, Theorem~1]{Mont}
(which a  more precise form of the celebrated Erd{\H o}s--Tur\'{a}n inequality).

\begin{lem}
\label{lem:ET small int}
Let $\gamma_1, \ldots, \gamma_M$ be a sequence of $M$ points of the unit interval $[0,1]$.
Then for any integer $K\ge 1$, and an interval $[\alpha, \beta] \subseteq [0,1]$,
we have
\begin{equation*}
\begin{split}
\# \{n =1, \ldots, M~:&~\gamma_n  \in [\alpha, \beta]\} - M(\beta - \alpha)\\
\ll \frac{M}{K} + &\sum_{k=1}^K \(\frac{1}{K} +\min\{\beta - \alpha, 1/k\}\)
\left|\sum_{n=1}^M \exp(2 \pi i k \gamma_n)\right|.
\end{split}
\end{equation*}
\end{lem}

To use Lemma~\ref{lem:ET small int} we also need an estimate on exponential sums
with polynomials,  which is essentially due to Weyl, see~\cite[Proposition~8.2]{IwKow}.

Let $\|\xi\| = \min\{|\xi - k|~:~k\in \Z\}$ denote the distance between
a real $\xi$ and the closest integer.

\begin{lem}
\label{lem:Weyl}
Let $f(X) \in \R[X]$ be a polynomial of degree $m\ge 2$ with the leading coefficient
$\vartheta \ne 0$.
Then
\begin{equation*}
\begin{split}
&\left|\sum_{n=1}^M \exp(2 \pi i f(n))\right|\\
&\quad \ll M^{1-m/2^{m-1}}  \(\sum_{-M < \ell_1 , \ldots,  \ell_{m-1}  < M}
\min\{M, \|\vartheta m!  \ell_1   \ldots   \ell_{m-1}\|^{-1}\}\)^{2^{1-m}}.
\end{split}
\end{equation*}
\end{lem}

\subsection{Integer points on curves and varieties}

We also need the following estimate
of Bombieri and Pila~\cite{BoPi} on
the number of integral points on polynomial curves.

\begin{lem}\label{lem:BombPila}
Let $\cC$ be an absolutely irreducible curve of degree $d\ge 2$
and $H\ge \exp(d^6)$. Then the number of integral points on $\cC$ and
inside of a square $[0,H]\times [0,H]$ does not exceed
$H^{1/d}\exp(12\sqrt{d\log H\log \log H})$.
\end{lem}

The following statement is a particular case of a more general result of
Wooley~\cite[Theorem 1.1]{Wool2}.

\begin{lem}
\label{lem:Wooley} The number of solutions of the system of diophantine equations
$$
x_1^j+\ldots +x_{8}^j=x_{9}^j+\ldots +x_{16}^j,\quad j=1,2,3
$$
in integers $x_i$ with $|x_i| \le M$,
$i=1, \ldots, 16$, is at most
$M^{10+o(1)}$.
\end{lem}

\begin{proof} Writing  $x_i = X_i -M-1$ with a positive integer $X_i \le 2M+1$,
$i=1, \ldots, 16$, after some trivial algebraic transformation we
see that the number of  solutions to the above equation is equal to
$J_{8,3}(2M+1)$. Since by  the result of
Wooley~\cite[Theorem~1.1]{Wool2} we have  $\kappa(3)\le 8$, the
bound~\eqref{eq:bound J} applies with $H =2M+1$.
\end{proof}

We note that Lemma~\ref{lem:Wooley} can be formulated in a more
general form with $\kappa(3)$ instead of $8$ variables on each side,
but this generalization (assuming possible improvements of the bound
$\kappa(3)\le 8$) does not affect our
main results.

\subsection{Congruences with many solutions}

The following result is used in the proofs of Theorems~\ref{thm:ConcentElliptic} and~\ref{thm:HyperellipticOptimal}.

\begin{lem}
\label{lem:deter}
Let $f,g\in \F_p[X]$ be two  polynomials of degrees $n$ and $m$ such that $m \nmid n$.
Assume that the integers $x_1,\ldots ,x_n$ are pairwise
distinct modulo $p$
and $y_1, \ldots ,y_n$ are arbitrary integers.
Then the  congruence
\begin{equation}
\label{eq:f=g}
f(x) \equiv g(y)\pmod p, \qquad 0 \le x,y < p,
\end{equation}
has at most $mn$ solutions with
\begin{equation}
\label{eq:matrix}
\det\(
\begin{matrix}x^n & x^{n-1} & \ldots & x & y\\x_1^n & x_1^{n-1} &
\ldots & x_1 & y_1\\ & & \ldots \\x_n^n & x_n^{n-1} & \ldots & x_n &
y_n
\end{matrix}\)\equiv 0 \pmod p.
\end{equation}
\end{lem}

\begin{proof}
Since $$
\det\(\begin{matrix}x_1^n & x_1^{n-1} &
\ldots & x_1 \\ & & \ldots \\x_n^n & x_n^{n-1} & \ldots & x_n
\end{matrix}\)=x_1\ldots x_n\prod_{1\le i<j\le n}(x_i-x_j)\not \equiv
0 \pmod p ,$$
we deduce that, for any $x$ and $y$, the last column
in~\eqref{eq:matrix} is a unique modulo $p$ linear
combination of the previous columns. In particular, for
every solution
$(x,y)$ to ~\eqref{eq:f=g} and~\eqref{eq:matrix}  we have
$y\equiv h(x) \pmod p$ for some nontrivial polynomial $h(X)
\in \F_p[X]$
that does not depend on $x$ and $y$.

Now we insert this into~\eqref{eq:f=g}. We observe that now the
right hand side of~\eqref{eq:f=g}, that is $g(h(x))$, is a nontrivial polynomial
 of degree $m\deg h$.
Thus, the congruence~\eqref{eq:f=g} is a nontrivial polynomial
congruence of degree $d$ with $n\le
d\le mn$. Therefore it has at most $mn$ solutions modulo $p$.
\end{proof}

\subsection{Background on Geometry of Numbers}

We recall that a lattice in $\R^n$ is an additive subgroup of $\R^n$
generated by $n$ linearly independent vectors.
Let $D$ be a symmetric convex body, that is, $D$ is a compact convex subset of
$\mathbb{R}^{n}$ with non-empty interior that is centrally symmetric with respect to $0$.
Then, for a lattice in $\Gamma\subseteq\R^n$
and $i=1,\ldots,n$, the $i$-th successive minimum
$\lambda_i(D,\Gamma)$
of the set $D$ with respect to the lattice $\Gamma$ is defined as
the minimal number $\lambda $ such that the set $\lambda  D$
contains
$i$ linearly independent vectors of the lattice $\Gamma$. In particular
$\lambda_1(D,\Gamma)\le\ldots\le\lambda_n(D,\Gamma)$. We recall the
following result given in~\cite[Proposition~2.1]{BHW} (see
also~\cite[Exercise~3.5.6]{TaoVu} for a simplified form that is
still enough for our purposes).

\begin{lem}
\label{lem:latp} We have,
$$\#(D\cap\Gamma)\le \prod_{i=1}^n \(\frac{2i}{\lambda_i(D,\Gamma)} + 1\).
$$
\end{lem}

Using that
$$\frac{2i}{\lambda_i(D,\Gamma)} + 1 \le (2i+1)\max\left
\{\frac1{\lambda_i(D,\Gamma)},1\right\}$$
and denoting, as usual, by $(2n+1)!!$ the product of all odd positive numbers
up to $2n+1$,  we derive:

\begin{cor} \label{cor:latpoints} We have,
$$\prod_{i=1}^n \min\{\lambda_i(D,\Gamma),1\} \le (2n+1)!!
(\#(D\cap\Gamma))^{-1}.$$
\end{cor}

\section{Proofs}

\subsection{Proof of Theorem~\ref{thm:ConcentElliptic}}

For the sake of brevity,
in this section we denote $I=I_{f}(M;R,S)$. We can assume that $I$ is large. We fix some  $L$ with
\begin{equation}\label{eq:1cong L}
 1\le L\le \frac{I}{20},
\end{equation}
to
be chosen later. By the pigeonhole principle,  there exists $Q$ such that the congruence
$$
y^{2}\equiv f(x)\pmod p, \qquad
Q+1\le x\le Q+M/L,\
S+1\le y\le S+M,
$$
has at least $I/L$ solutions.
Since there are at most two solutions to the above congruence
with the same value of $x$, by the pigeonhole principle,
there exists an interval of length $20M/I$  containing at least
$10$ solutions $(x,y)$ with pairwise distinct values $x$. Let
$x_0$ be the first of these values and let $(x_0,y_0)$
be the corresponding solution. It is
clear that $I/L$ is bounded by the number of solutions of
\begin{equation*}
\begin{split}
  (y_0+y)^2&\equiv f(x_0+x) \pmod p,\\
  -M/L\le x&\le M/L,\quad -M\le y\le M,
\end{split}
\end{equation*}
which is equivalent to
\begin{equation}
\label{eqn:1c}
\begin{split}
&y^2\equiv c_3x^3+c_2x^2+c_1x+c_0y\pmod p,\\
&-M/L\le x\le M/L,\quad -M\le y\le M,
\end{split}
\end{equation}
with $(c_3,p)=1$.
Besides, there are at least $10$ solutions $(x,y)$ with $x$ pairwise
distinct and such that
$0 \le x\le 20M/I$. From
these $10$ values we fix $3$ solutions $(x_1,y_1), (x_2,y_2),(x_3,y_3)$ and rewrite the congruence~\eqref{eqn:1c}  in the matrix form
\begin{equation}\label{eq:cong 12}
\left (\begin{matrix}x^3 & x^2 & x & y\\
x_3^3 & x_3^2 & x_3 & y_{3}\\
x_2^3 & x_2^2 & x_2 & y_2\\
x_1^3 & x_1^2 & x_1 & y_1\\
\end{matrix}\right )\(\begin{matrix}c_3\\
c_2\\c_1 \\c_0  \end{matrix}\)\equiv
\(\begin{matrix}y^2\\y_3^2\\y_2^2 \\y_1^2
\end{matrix}\)\pmod p.
\end{equation}
By Lemma~\ref{lem:deter},
we know that at most $6$
pairs $(x,y)$, with $x$ pairwise distinct, satisfy both the
congruence~\eqref{eq:cong 12} and the congruence
$$
\left |\begin{matrix} x^{h} &\ldots & x & y\\ x_{h}^{h}
&\ldots & x_{h} & y_{h}\\  & \ldots & \\x_1^{h} &\ldots & x_1 &
y_1
\end{matrix}\right |\equiv 0\pmod p.
$$
Since there are at least $10$ solutions to~\eqref{eq:cong 12}, for one of
them, say $(x_4,y_4)$, we have
$$
\Delta=\left |\begin{matrix}x_4^3 & x_4^2 & x_4 & y_4\\
x_3^3 & x_3^2 & x_3 & y_{3}\\
x_2^3 & x_2^2 & x_2 & y_2\\
x_1^3 & x_1^2 & x_1 & y_1\\
\end{matrix}\right |\not \equiv 0\pmod p.
$$

Note that $1\le |\Delta |\ll (M/I)^6M$. Now we solve the
system of congruences
\begin{equation}
\label{eq:cong 13}\left (\begin{matrix}x_4^3 & x_4^2 & x_4 & y_4\\
x_3^3 & x_3^2 & x_3 & y_{3}\\
x_2^3 & x_2^2 & x_2 & y_2\\
x_1^3 & x_1^2 & x_1 & y_1\\
\end{matrix}\right )\(\begin{matrix}c_3\\
c_2\\c_1 \\c_0  \end{matrix}\)\equiv
\(\begin{matrix}y_4^2\\y_3^2\\y_2^2 \\y_1^2
\end{matrix}\)\pmod p
\end{equation}
with respect to $(c_3, c_2, c_1, c_0)$. We write $\Delta_j$ for
the determinant of the matrix on the left hand side where we have
substituted the column $j$ by the vector $(y_4^2,y_3^2,y_2^2,y_1^2)$. With this notation we have that
$$
c_j\equiv \Delta_{4-j}\Delta^* \pmod p,\quad j=0,\ldots 3,
$$
where $\Delta^*$ is defined by $\Delta\Delta^*\equiv 1 \pmod p$,
and the congruence~\eqref{eqn:1c} is equivalent to
$$
\Delta_1x^3+\Delta_2x^2+\Delta_3x+\Delta_4y- \Delta y^2\equiv 0\pmod p.
$$
In particular, since, as we have noticed, $c_3\not\equiv 0\pmod p$, we have that
$\Delta_1\not\equiv 0\pmod p$. We can
write this congruence as an equation over $\Z$:
\begin{equation}\label{eq:1integer}
\Delta_1x^3+\Delta_2x^2+\Delta_3x+\Delta_4y- \Delta y^2=pz,
\qquad (x,y,z)\in \Z^3.
\end{equation}

We  can easily check that
$$
|\Delta_4|\ll (M/I)^6M^2
$$
and
$$
|\Delta_j|\ll (M/I)^{2+j}M^3,\qquad j=1,2,3.
$$
Thus, collecting the above estimates and taking into account $L\ll I$, we derive
\begin{equation*}
\begin{split}
|z|&\ll  \frac{1}{p} \(|\Delta_1|(M/L)^3+|\Delta_2|(M/L)^2+|\Delta_3|(M/L)
+|\Delta_4|M+|\Delta|M^2\)\\
&\ll   \frac{M^3}{p}\(\frac{M^6}{I^3L^3}+\frac{M^7}{I^4L^2}+\frac{M^6}{I^5L}+\frac{M^6}{I^6}\) \ll   \frac{M^9}{pI^3L^3}.
\end{split}
\end{equation*}
Since $\Delta_1\not=0,\,\Delta\not=0$, for each $z$, the curve~\eqref{eq:1integer} is absolutely irreducible,
and thus by Lemma~\ref{lem:BombPila}
it contains at most $M^{1/3+o(1)}$ integer points $(x,y)$
with $|x|,|y|\le M$. Hence
$$
\frac{I}{L}\le M^{1/3+o(1)}
\(1+\frac{M^9}{pI^3L^3}\)
$$
for any $L$ satisfying~\eqref{eq:1cong L}. This implies, that
\begin{equation}
\label{eq:1prelim}
I\le LM^{1/3+o(1)}+\frac{M^{7/3}}{p^{1/4}L^{1/2}}.
\end{equation}

If $M<10p^{1/8}$, then we take $L=1$ and derive from~\eqref{eq:1prelim} that
$$
I\le M^{1/3+o(1)}+\frac{M^{7/3+o(1)}}{p^{1/4}}\le M^{1/3+o(1)}.
$$

Let now $M>10p^{1/8}$. We can assume that $I>M^{5/3}p^{-1/6}$, as otherwise there is nothing to prove. Then we take $L=\fl{M^{4/3}p^{-1/6}}$ and note
that the condition~\eqref{eq:1cong L} is satisfied.
Thus, we derive from~\eqref{eq:1prelim} that
$$
I\le LM^{1/3+o(1)}+\frac{M^{7/3+o(1)}}{p^{1/4}L^{1/2}}\le M^{5/3+o(1)}p^{-1/6}
$$
and the result follows.

\subsection{Proof of Theorem~\ref{thm:ConcentEllipticWiderRange}}

Clearly we can assume that
\begin{equation}
\label{eq:Big M}
M > p^{5/23}
\end{equation}
as otherwise
$$
(M^3/p)^{1/16} M < \frac{M^{5/3+o(1)}}{p^{1/6}}
$$
and the result follows from Theorem~\ref{thm:ConcentElliptic}. We can also assume that $M=o(p^{1/3})$.

We fix one solution $(x_0,y_0)$ to the congruence~\eqref{eq:cong y2=f(x)} and by making the change of variables $(x,y) \mapsto (x-x_0, y-y_0)$, we  see that it is enough to
study a congruence of the form
\begin{equation}
\label{eq:ellipticAfterShifts}
y^2-c_0y\equiv c_3x^3+c_2x^2+c_1x\pmod p,\quad |x|,|y|\le M.
\end{equation}
Let $\cW$ be the set of pairs $(x,y)$ that satisfy~\eqref{eq:ellipticAfterShifts},
and by $\cX$ we denote the set of $x$ for which $(x,y)\in \cW$ for some $y$.
Let
$$
\rho = \frac{\#\cX}{M}.
$$

We now fix some $\varepsilon > 0$ and  assume that
\begin{equation}
\label{eq:large rho1} \rho  \ge  (M^{3}/p)^{1/16} M^\varepsilon.
\end{equation}
We also assume that $M$ is sufficiently large. In view of~\eqref{eq:Big M} and~\eqref{eq:large rho1}, we also have
\begin{equation}
\label{eq:large rho2}
\rho > M^{-1/10}.
\end{equation}

For  $\vartheta >0$ we define  the intervals
$$
I_{\nu,\vartheta}=[-\vartheta M^\nu, \vartheta M^\nu],
\qquad  \nu =1,2,3,
$$
which we treat as intervals in $\F_p$, that is, sets
of residues modulo $p$ of several consecutive integers.

We now consider the set
$$
\cS\subseteq  I_{1,8}\times I_{2,8}\times I_{3,8}
$$
of all triples
\begin{equation}
\label{eq:triples} \vec{s}
\equiv
(x_1+\ldots +x_{8}, \, x_1^2+\ldots
+x_{8}^2, \, x_1^3+\ldots +x_{8}^3)\pmod p,
\end{equation}
where $x_i$, $i=1, \ldots, 8$, independently run through the set
$\cX$. We observe that the system of congruences
\begin{equation}
\label{eq:congrWooley} x_1^j+\ldots +x_{8}^j\equiv x_{9}^j+\ldots
+x_{16}^j \pmod p,\qquad j=1,2,3,
\end{equation}
has   at most $M^{10+o(1)}$ solutions in integers $x_i,y_i$ with
$|x_i|, |y_i|\le M$. Indeed, since
$M=o(p^{1/3})$,
the
above congruence is converted to the system of diophantine equations
$$
x_1^j+\ldots +x_{8}^j=x_{9}^j+\ldots +x_{16}^j,\quad j=1,2,3,
$$
which by Lemma~\ref{lem:Wooley} has at most $M^{10+o(1)}$ solutions
in integers $x_i$ with $|x_i| \le M$,  $i=1, \ldots, 16$. Therefore,
the congruence~\eqref{eq:congrWooley} has at most $M^{10+o(1)}$
solutions in $x_i \in \cX$,  $i=1, \ldots, 16$,  as well. Thus,
collecting elements of the set $\cX^{8}$ that correspond to
the same
vector $\vec{s}$ given by~\eqref{eq:triples} and denoting the number
of such representations by $N(\vec{s})$, by the Cauchy inequality,
we obtain
$$
(\#\cX)^{8} = \sum_{\vec{s} \in \cS} N(\vec{s}) \le \(\#\cS
\sum_{\vec{s} \in \cS} N(\vec{s})^2\)^{1/2} \le \(\#\cS
M^{10+o(1)}\)^{1/2}.
$$
Thus
$$
\#\cS\ge \frac{(\#\cX)^{16}}{M^{10+o(1)}}=\rho^{16}M^{6+o(1)}.
$$
Hence, there exist
at least $\rho^{16}M^{6+o(1)}$ triples
$$
(z_1,z_2,z_3)\in I_{1,8} \times I_{2,8} \times I_{3,8}
$$
such that
$$
c_3z_3+c_2z_2+c_1z_1\equiv \tz_2-c_0\tz_1\pmod p
$$
for some $\tz_2\in I_{2,8}$ and $\tz_1\in I_{1,8}$.
In particular we have that the congruence
\begin{equation*}
\begin{split}
c_3z_3+c_2z_2+\tz_2 &+c_1z_1+c_0\tz_1\equiv 0\pmod p,\\
(z_1,\tz_1, z_2,\tz_2,z_3)&\in I_{1,8}\times I_{1,8} \times I_{2,8}\times I_{2,8}\times
I_{3,8},
\end{split}
\end{equation*}
has a set of solutions $\cS$ with
\begin{equation}
\label{eq:S large}
\# \cS\ge\rho^{16}M^{6+o(1)}.
\end{equation}

The rest of the proof is based on the ideas from~\cite{BGKS}.

We define the lattice
\begin{equation*}
\begin{split}
\Gamma = \{(X_2,X_3,\tX_2,X_1,\tX_1)&\in\Z^5~:\\
~
X_2+c_3X_3+c_2&\tX_2 + c_1X_1 + c_0\tX_1\equiv 0 \pmod p\}
\end{split}
\end{equation*}
and the body
\begin{equation*}
\begin{split}
D = \{(x_2,x_3,\tx_2,&x_1,\tx_1)\in\R^5~:\\
&|x_1|, |\tx_1|\le 8 M, \ |x_2|, |\tx_2|\le 8 M^{2},\ |x_3|\le 8 M^{3}\}.
\end{split}
\end{equation*}

We see from~\eqref{eq:S large} that
$$
\#\(D \cap \Gamma\) \ge \rho^{16}M^{6+o(1)}.
$$
Therefore, by Corollary~\ref{cor:latpoints}, the successive minima
$\lambda_i=\lambda_i(D,\Gamma)$, $i=1,\ldots,5$, satisfy the
inequality
\begin{equation}
\label{eq:ineqlambda}
 \prod_{i=1}^5\min\{1,\lambda_i\}\ll  \rho^{-16}M^{-6+o(1)}.
\end{equation}

From the definition of $\lambda_i$ it follows that there are five linearly independent vectors
\begin{equation}
\label{eq:vec5}
\vec{v}_i = (v_{2,i},v_{3,i},\tv_{2,i},v_{1,i},\tv_{1,i}) \in \lambda_i D\cap\Gamma, \quad i=1,\ldots,5.
\end{equation}
Indeed, first we choose a nonzero vector $\vec{v}_1 \in  \lambda_1 D\cap\Gamma$ and then assuming that for $1 \le i\le 5$ the
vectors  $\vec{v}_1, \ldots, \vec{v}_{i-1}$ are chosen, we choose $\vec{v}_i$ as of the
$i$ linearly independent vectors  $\vec{v} \in  \lambda_i D\cap\Gamma$ that is not in the
linear space generated by  $\vec{v}_1, \ldots, \vec{v}_{i-1}$.

We now note that
$$
\lambda_3 < 1.
$$
Indeed, otherwise from~\eqref{eq:ineqlambda} we obtain
$$
\min\{1,\lambda_1^2\}\le \min\{1,\lambda_1\} \min\{1,\lambda_2\}
\le  \rho^{-16}M^{-6+o(1)}.
$$
Thus recalling~\eqref{eq:large rho2} we see that
$$
\lambda_1 \le \frac{1}{10M^2}
$$
Then the vector
$\vec{v}_1$ must have $v_{2,1}= \tv_{2,1} = v_{1,1} =\tv_{1,1}=0$. In turn
this implies that $ v_{3,1} \equiv 0 \pmod p$ and since we assumed that $M = o(p^{1/3})$,
we obtain $v_{3,1} = 0$, which contradicts the condition that $\vec{v}_1$ is a nonzero vector.

We consider separately the following
four cases.

{\it Case~1\/}: $\lambda_5\le 1$. Then
by~\eqref{eq:ineqlambda}, we have
$$ \prod_{i=1}^5 \lambda_i  \le  \rho^{-16}M^{-6+o(1)}.
$$

We now consider the determinant $\Delta$ of the $5\times 5$ matrix that is formed
by the  vectors~\eqref{eq:vec5}.
It follows that
$$\Delta \ll M^{2+3+2+1+1} \prod_{i=1}^5 \lambda_i  \le \rho^{-16} M^{3+o(1)},
$$
which, by our assumption~\eqref{eq:large rho1}, implies that $|\Delta|<p$.
On the other hand, since $\vec{v}_i\in\Gamma$, we have $\Delta \equiv 0 \pmod p$, thus $\Delta = 0$
provided that $p$ is large enough,
which contradicts the linear independence
of the vectors in~\eqref{eq:vec5}.
Thus this case is
impossible.

{\it Case~2\/}: $\lambda_4\le 1$, $\lambda_5> 1$.
Let
$$
V = \(\begin{array}{cccc}
    v_{3,1}&\tv_{2,1}&v_{1,1}&\tv_{1,1}\\
    v_{3,2}&\tv_{2,2}&v_{1,2}&\tv_{1,2}\\
     v_{3,3}&\tv_{2,3}&v_{1,3}&\tv_{1,3}\\
    v_{3,4}&\tv_{2,4}&v_{1,4}&\tv_{1,4}\\
  \end{array}\), \quad
 \vec{w} = \(
  \begin{array}{c}
 -v_{2,1}\\
- v_{2,2}\\
- v_{2,3}\\
- v_{2,4}\\
  \end{array}
\),\quad
\vec{c} = \(
  \begin{array}{c}
   c_3\\
  c_2\\
c_1\\
c_0\\
  \end{array}
\).
$$

We have
$$
V \vec{c} \equiv \vec{w}\pmod p.
$$
Let
$$
\Delta= \det  V
$$
and let $\Delta_j$ be the determinants of the matrix obtained  by
replacing the $j$-th column of $V$ by $\vec{w}$, $j =1, \ldots, 4$.

Recalling~\eqref{eq:ineqlambda}, we have
\begin{equation}
\label{eq:T2Case2Delta}
|\Delta| \ll \lambda_1\lambda_2 \lambda_3 \lambda_4 M^{3+2+1+1}
\le \rho^{-16} M^{1+o(1)}
\end{equation}
and similarly
\begin{equation}
\label{eq:T2Case2Deltas}
\begin{split}
|\Delta_1| & \le \rho^{-16} M^{o(1)}, \qquad  |\Delta_2|\le \rho^{-16} M^{1+o(1)},\\
 |\Delta_3| &\le \rho^{-16} M^{2+o(1)}, \qquad  |\Delta_4|\le \rho^{-16} M^{2+o(1)}
\end{split}
\end{equation}
Note that, in view of~\eqref{eq:large rho1}, in particular we have
$$
|\Delta|, |\Delta_j| <p, \quad j =1, \ldots, 4.
$$
If $\Delta\equiv 0\pmod p$ then since $\vec{c}$ is nonzero modulo $p$ we also have
$\Delta_j\equiv 0\pmod p$, $j =1, \ldots, 4$, implying that $\Delta=0,\, \Delta_j=0$.
Then the  matrix formed by $\vec{v}_1, \ldots, \vec{v_4}$
is of rank at most $3$,
which contradicts
their linear independence.
Therefore $\Delta\not \equiv 0 \pmod p$ and thus we have
$$
c_i \equiv \frac{\Delta_{4-i}}{\Delta} \pmod p,
\qquad i = 0, 1, 2, 3.
$$
Since $c_3\not\equiv 0\pmod p$, we have $\Delta_1\not=0$. We now substitute this in~\eqref{eq:ellipticAfterShifts} and get that
$$
\Delta y^2-\Delta_4 y\equiv \Delta_1x^3+\Delta_2x^2+\Delta_3x\pmod p,\quad |x|,|y|\le M.
$$
We see from~\eqref{eq:large rho1}, \eqref{eq:T2Case2Delta} and~\eqref{eq:T2Case2Deltas}
that for sufficiently large $M$ the expressions on
both sides
are less than $p/2$, implying the equality
$$
\Delta y^2-\Delta_4 y=\Delta_1x^3+\Delta_2x^2+\Delta_3x,\qquad |x|,|y|\le M.
$$
Now we use Lemma~\ref{lem:BombPila} and conclude that the number of solutions is at most $M^{1/3+o(1)}$.

{\it Case~3\/}: $\lambda_3\le (10M)^{-1}$, $\lambda_4> 1$. By~\eqref{eq:ineqlambda}, we have
$$ \prod_{i=1}^3 \lambda_i  \le  \rho^{-16}M^{-6+o(1)}.
$$
Since $\lambda_3\le (10M)^{-1}$, we also have
\begin{equation}
\label{eq:Case3}
\vec{v}_i = (v_{2,i},v_{3,i},\tv_{2,i},0,0),\quad i=1,2,3.
\end{equation}
In particular,
$$
\(\begin{array}{cccc}
    v_{2,1}&v_{3,1}&\tv_{2,1}\\
    v_{2,2}&v_{3,2}&\tv_{2,2}\\
     v_{2,3}&v_{3,3}&\tv_{2,3}\\
  \end{array}\)\(
  \begin{array}{c}
   1\\
  c_3\\
c_2\\
  \end{array}
\)\equiv \(
  \begin{array}{c}
   0\\
  0\\
0\\
  \end{array}
\)\pmod p.
$$
Thus, for the determinant
$$
\Delta=\det \(\begin{array}{cccc}
    v_{2,1}&v_{3,1}&\tv_{2,1}\\
    v_{2,2}&v_{3,2}&\tv_{2,2}\\
     v_{2,3}&v_{3,3}&\tv_{2,3}\\
  \end{array}\)
$$
we have
$$
\Delta \equiv 0\pmod p.
$$
On the other hand, from~\eqref{eq:large rho2} we derive that
$$
|\Delta|\ll \lambda_1\lambda_2\lambda_3 M^7<\frac{M^{1+o(1)}}{\rho^{16}}< M^{2.6+o(1)}.
$$
Hence, $\Delta=0$, which together with~\eqref{eq:Case3} implies that the vectors $\vec{v}_1,\vec{v}_2,\vec{v}_3$ are linearly dependent,  which is impossible.

{\it Case~4\/}: $(10M)^{-1}< \lambda_3 \le 1 $, $\lambda_4> 1$. By~\eqref{eq:ineqlambda}, we have
$$
\prod_{i=1}^3 \lambda_i  \le  \rho^{-16}M^{-6+o(1)}
$$
and since $\lambda_3 > (10M)^{-1}$, we obtain
$$
\lambda_1\lambda_2< \rho^{-16}M^{-5+o(1)}.
$$
We again note that $\lambda_1 > (10M^2)^{-1}$, as otherwise $\vec{v}_1$ must have $v_{2,1}= \tv_{2,1} = v_{1,1} =\tv_{1,1}=0$. In turn
this implies that $ v_{3,1} \equiv 0 \pmod p$ and since we assumed that $M = o(p^{1/3})$,
we obtain $v_{3,1} = 0$, which contradicts the condition that $\vec{v}_1$ is a nonzero vector.

Since $\lambda_1 > (10M^2)^{-1}$ and $\rho>M^{-1/10}$, we get that $\lambda_2<(10M)^{-1}$. Thus, we have
$$
\vec{v}_i=(v_{2,i},v_{3,i},\tv_{2,i},0,0),\qquad i=1,2.
$$
Next,
$$
\(\begin{array}{cccc}
    v_{3,1}&\tv_{2,1}\\
    v_{3,2}&\tv_{2,2}\\
  \end{array}\)\(
  \begin{array}{c}
  c_3\\
c_2\\
  \end{array}
\)\equiv \(
  \begin{array}{c}
   -v_{2,1}\\
  -v_{2,2}\\
  \end{array}
\)\pmod p.
$$
Now we observe that
\begin{equation}
\label{eq:Case4Delta}
\Delta=\det\(\begin{array}{cccc}
    v_{3,1}&\tv_{2,1}\\
    v_{3,2}&\tv_{2,2}\\
  \end{array}\)\ll \lambda_1\lambda_2 M^{5}<\frac{M^{o(1)}}{\rho^{16}}.
\end{equation}
Furthermore,
\begin{equation}
\label{eq:Case4Delta1}
\Delta_1=\det\(\begin{array}{cccc}
    -v_{2,1}&\tv_{2,1}\\
    -v_{2,2}&\tv_{2,2}\\
  \end{array}\)\ll \lambda_1\lambda_2 M^{4}<\frac{M^{-1+o(1)}}{\rho^{16}},
\end{equation}
and
\begin{equation}
\label{eq:Case4Delta2}
\Delta_2=\det\(\begin{array}{cccc}
    v_{3,1}&-v_{2,2}\\
    v_{3,2}&-v_{2,2}\\
  \end{array}\)\ll \lambda_1\lambda_2 M^{5}<\frac{M^{o(1)}}{\rho^{16}}.
\end{equation}
In particular, $|\Delta|, |\Delta_1|, |\Delta_2|<p$. Therefore, if $\Delta\equiv 0\pmod p$, then
$\Delta_1\equiv\Delta_2\equiv 0\pmod p$ and we see that $\Delta=\Delta_1=\Delta_2=0$. Thus, in this case the rank of the matrix
formed with vectors $\vec{v}_1, \vec{v}_2$ is at most $1$, which contradicts
the linear independence of the vectors $\vec{v}_1, \vec{v}_2$.

Hence, $\Delta\not\equiv 0\pmod p$ and we get that
$$
c_3\equiv \frac{\Delta_1}{\Delta}\pmod p,\qquad c_2\equiv \frac{\Delta_2}{\Delta}\pmod p.
$$
We now substitute this in~\eqref{eq:ellipticAfterShifts} and get that
$$
\Delta y^2-a_0 y\equiv \Delta_1x^3+\Delta_2x^2+b_0x\pmod p,\qquad |x|,|y|\le M,
$$
for some integers $a_0,b_0$. We observe that the condition $c_3\not\equiv 0\pmod p$ implies that $\Delta_1\not=0$.

Let now
$$
T=\fl{\(\frac{p}{M}\)^{1/3}\rho^{16/3}}.
$$
Note that $M^{2/3}<T<T^2<p/2$. By the pigeonhole principle, there exists a positive integer $1\le t_0\le T^2+1$ such that
$$
|(t_0a_0)_p|\le \frac{p}{T},\qquad |(t_0b_0)_p|\le \frac{p}{T},
$$
where $(x)_p$ is the element of the residue class $x\pmod p$ with the least absolute value,
see also~\cite[Lemma~4]{CSZ}. Hence
$$
t_0\Delta y^2-(t_0a_0)_p y\equiv t_0\Delta_1x^3+t_0\Delta_2x^2+(t_0b_0)_p x\pmod p,\quad |x|,|y|\le M.
$$
By~\eqref{eq:Case4Delta},~\eqref{eq:Case4Delta1},~\eqref{eq:Case4Delta2}, the absolute value of the both hand side is bounded by $pM^{1+o(1)}T^{-1}$. Thus, we get
$$
t_0\Delta y^2-(t_0a_0)_p y= t_0\Delta_1x^3+t_0\Delta_2x^2+(t_0b_0)_p x+pz,
$$
where
$$
|x|,|y|\le M, \quad |z|<M^{1+o(1)}T^{-1}.
$$
Now we use Lemma~\ref{lem:BombPila} and conclude that the number of solutions is at most
\begin{equation*}
\begin{split}
\(\frac{M}{T}+1\)M^{1/3+o(1)}&<\(\frac{M^{4/3}}{p^{1/3}}\rho^{-16/3}+1\)M^{1/3+o(1)}\\
&<M^{2/3+o(1)}<\(\frac{M^3}{p}\)^{1/16}M.
\end{split}
\end{equation*}

Since $\varepsilon >0$ is arbitrary, the result now follows.
\subsection{Proof of Theorem~\ref{thm:ConcentDegf>3}}

Let $\cX$ be the set of integers $x\in [R+1,R+M]$ such that the congruence~\eqref{eq:cong y2=f(x)} is satisfied for some integer $y\in [S+1, S+M].$
In particular, letting $X=\#\cX$ we have
\begin{equation}
\label{eq:II0}
I_{f}(M;R,S)  \le 2X.
\end{equation}

Fix some integer $k\ge 1$ and consider the set
$$
\cY_k=\{y_1^2+\ldots +y_k^2\pmod p~:~S+1\le y_i\le S+M, \ i=1, \ldots, k\}.
$$
By making the change of variables
$y_i=S+z_i$, $i=1, \ldots, k$, we observe that
\begin{equation*}
\begin{split}
\cY_k= \{z_1^2+\ldots  +z_k^2+2S(z_1+\ldots &+z_k)+kS^2\pmod p~:\\
&
1\le z_i\le M, \ i=1, \ldots, k\}.
\end{split}
\end{equation*}
In particular,
$$\# \cY_k\le \# \left \{r+2Ss+kS^2~:~1\le r\le kM^2,\ 1\le s\le kM\right \}\le k^2M^3.$$
For any $(x_1,\ldots,x_k)\in \cX^k$ there exists $\lambda\in \cY_k$ such that
$$
f(x_1)+\ldots+f(x_k)\equiv \lambda\pmod p.
$$
Thus,
$$
X^k \le \sum_{\lambda\in \cY_k}r(\lambda)
$$
where
\begin{equation*}
\begin{split}
r(\lambda)= \# \{(x_1,\ldots, x_k)\in &[R+1,R+M]^k~:\\
&~f(x_1)+\ldots
+f(x_k)\equiv \lambda\pmod p\} .
\end{split}
\end{equation*}
Using  the Cauchy inequality, we derive
$$
X^{2k}\le \# \cY_k \sum_{\lambda\in \cY_k}r^2(\lambda)\le
k^2M^3T_k(R,M),
$$
where $T_k(R;M)$ is the number of solutions of
\begin{eqnarray*}
f(x_1)+\ldots +f(x_k)&\equiv &f(x_{k+1})+\ldots +f(x_{2k})\pmod p,
\\ (x_1,\ldots,x_{2k})&\in& [R+1,R+M]^{2k}.
\end{eqnarray*}
The quantity $T_k(R;M)$ has been defined and estimated in~\cite{CGOS} for $R=0$ but
making a change of variables, it
is clear that the same bound holds for any $R$.
In particular, it is proved  in~\cite{CGOS} that
$$
T_k(R;M)\ll \left (M^m/p+1\right )M^{m(m-1)/2}J_{k,m}(M),
$$
where, as before, $J_{k,m}(M)$ is the number of solutions of the system
of equations~\eqref{eq:System} with $H = M$.

Taking $k = \kappa(m)$
so that the bound~\eqref{eq:bound J} holds, we derive
\begin{equation*}
\begin{split}
X^{2k}& \le M^3\left (M^m/p+1\right )M^{m(m-1)/2}M^{2k-m(m+1)/2+o(1)}\\
& \le \left (M^m/p+1\right )M^{2k+3-m+o(1)}
\end{split}
\end{equation*}
and obtain
$$
X\le M(M^3/p)^{1/2\kappa +o(1)}+M^{1-(m-3)/2\kappa+o(1)},
$$
which together with~\eqref{eq:II0} concludes the proof.

\subsection{Proof of Theorem~\ref{thm:Polynomial Map}}

Let $J = J_f(M;R,S)$.

Without loss of generality we can assume that
$$
0 \le M+1< M+S < p.
$$
Applying Lemma~\ref{lem:ET small int} to the sequence
of fractional parts $\gamma_n = \{f(n)/p\}$, $n =1, \ldots, M$,
with
$$\alpha = (S+1)/p, \qquad \beta = (S+M+1)/p, \qquad K = \fl{p/M},
$$
so that we have
$$
\frac{1}{K} +\min\{\beta - \alpha, 1/k\} \ll \frac{M}{p}
$$
for $k =1, \ldots, K$, we derive
$$
J \ll \frac{M^2}{p} + \frac{M}{p} \sum_{k=1}^K
\left|\sum_{n=1}^M \exp(2 \pi i k f(n)/p)\right|.
$$
Therefore, by Lemma~\ref{lem:Weyl}, we have
\begin{equation*}
\begin{split}
J \ll \frac{M^2}{p} &+  \frac{M^{2-m/2^{m-1}}}{p}   \\
  \times &\sum_{k=1}^K
\(\sum_{-M < \ell_1 , \ldots,  \ell_{m-1}  < M}
\min\left\{M, \left \|\frac{a}{p}m! k \ell_1   \ldots
 \ell_{m-1}\right\|^{-1}\right\}\)^{2^{1-m}}.
\end{split}
\end{equation*}
Now, separating the contribution from the terms with $ \ell_1   \ldots
 \ell_{m-1} = 0$  we obtain
$$
J \ll \frac{M^2}{p} +   \frac{M^{2-m/2^{m-1}}}{p} K(M^{m-1})^{2^{1-m}}
+  \frac{M^{2-m/2^{m-1}}}{p} W ,
$$
where
$$
W = \sum_{k=1}^K
\(\sum_{0 < |\ell_1|, \ldots,  |\ell_{m-1}| < M}
\min\left\{M, \left \|\frac{a}{p}m! k \ell_1   \ldots
 \ell_{m-1}\right\|^{-1}\right\}\)^{2^{1-m}}.
$$

Hence, recalling the choice of $K$, we derive
\begin{equation}
 \label{eq:J and W}
J \ll  \frac{M^2}{p} +     M^{1-1/2^{m-1}} +  \frac{M^{2-m/2^{m-1}}}{p}  W.
\end{equation}
The H\"{o}lder inequality implies the bound
 \begin{equation*}
\begin{split}
W^{2^{m-1}} \ll K^{2^{m-1}-1}
& \sum_{k=1}^K  \\
&\sum_{0 < |\ell_1|, \ldots,  |\ell_{m-1}| < M}
\min\left\{M, \left \|\frac{a}{p}m! k \ell_1   \ldots
 \ell_{m-1}\right\|^{-1}\right\}.
\end{split}
\end{equation*}
Collecting together the terms with the same value of
$z = m! k \ell_1   \ldots  \ell_{m-1}$ and recalling the well-known bound on the divisor function,
we conclude that
$$
W^{2^{m-1}} \ll K^{2^{m-1}-1} p^{o(1)}
\sum_{|z| < m!KM^{m-1}}
\min\left\{M, \left \|\frac{a}{p}z \right\|^{-1}\right\}.
$$
Since the sequence $ \|am/p \|$ is periodic with
period $p$, we see that
 \begin{equation*}
\begin{split}
W^{2^{m-1}} & \ll K^{2^{m-1}-1} p^{o(1)} \frac{KM^{m-1}}{p}
\sum_{z=1}^p
\min\left\{M, \left \|\frac{a}{p}z \right\|^{-1}\right\}\\
& \ll K^{2^{m-1}-1} p^{o(1)} \frac{KM^{m-1}}{p}
\(M +
\sum_{z=1}^p   \left \|\frac{z}{p} \right\|^{-1} \)\\
& \ll K^{2^{m-1}} M^{m-1}p^{o(1)}  .
 \end{split}
\end{equation*}
Thus, recalling the choice of $K$, we derive
$$
W \le K  M^{(m-1)/2^{m-1}}p^{o(1)} \le   M^{(m-1)/2^{m-1}-1}p^{1+o(1)},
$$
which after the substitution in~\eqref{eq:J and W} concludes
the proof.

\subsection{Proof of Corollary~\ref{cor:HyperellipticAnyRange}}

Assume that  $H=\Hb$ for some vector $\vb=(b_0,\ldots ,b_{2g-1}) \in \F_p^{2g}$.
We recall that all components of any vector $\va \in \fB$ are
non-zero
modulo $p$.
Hence, $b_0 \in \F_p^*$ and  we see from~\eqref{eq:isom} (combinig
the equations with
$i=2g + 1 - h$
and $i=2g-1$) that
\begin{equation}\label{eq:cong 1}
\begin{split}
a_{2g-1}^{h}\equiv & \lambda  a_{g + 1 - h}^{2}\pmod
p, \\
R_{g + 1 - h}+1\le & a_{g + 1 - h}\le  R_{g + 1 - h}+M,\\
 R_{2g-1}+1\le&
a_{2g-1}\le R_{2g-1}+M,
\end{split}
\end{equation}
where
\begin{equation}
\label{eq:lambda}
\lambda=b_{2g-1}^{h}/b_{g + 1 - h}^{2}.
\end{equation}
 We also observe that
$$
\alpha^2=b_{2g-1}/a_{2g-1}.
$$
Thus, each
solution $(a_{g + 1 - h},a_{2g-1})$ of~\eqref{eq:cong 1} determines the value of
$\alpha^2$ and  therefore, all other  values of $a_0,a_1,\ldots ,a_{2g-1}$.

Thus we have seen that $N(H;\fB)\le T$, where $T$ is the number of
solutions $(x,y)$ of the congruence
\begin{equation}
\label{eq:xh y2}
x^{h}\equiv \lambda y^{2}\pmod p, \quad
R+1\le x\le R+M, \
S+1\le y\le S+M,
\end{equation}
where $R=R_{g + 1 - h}$, $S=R_{2g-1}$ and $\lambda$ is given by~\eqref{eq:lambda}.

We now observe that the congruence~\eqref{eq:xh y2} taken with $h = 4$,
which is admissible for $g \ge 2$, implies
$$
x^2\equiv \mu y\pmod p, \quad
R+1\le x\le R+M, \
S+1\le y\le S+M,
$$
where $\mu$ is one of the two square roots of $\lambda$ (we recall that $g\ge 2$).
Applying Theorem~\ref{thm:Polynomial Map} with a quadratic polynomial $f$, we immediately obtain
the desired result.

\subsection{Proof of Theorem~\ref{thm:HyperellipticOptimal}}

As in the proof of of Corollary~\ref{cor:HyperellipticAnyRange}
we let $H=\Hb$ for some $\vb=(b_0,\ldots ,b_{2g-1}) \in \F_p^{2g}$.

We can   assume that  $M<p^{1/4}$ as otherwise the results
are weaker
than the trivial upper bound $N(H;\fB)\ll M$.

Also we can assume that $T>M^{1/h}$, where, as before,
$T$ is the number of solutions $(x,y)$ to the
congruence~\eqref{eq:xh y2} as otherwise there is nothing to prove.

We follow the proof of Theorem~\ref{thm:ConcentElliptic}. We fix some  $L$ with
\begin{equation}\label{eq:cong L}
 1\le L\le \frac{T}{8(h+1)},
\end{equation}
to
be chosen later. Note that if $T<16g+16$ there is nothing to
prove. Thus, there exists $Q$ such that the congruence
$$
x^h\equiv \lambda y^{2}\pmod p, \qquad
Q\le x\le Q+M/L,\
S+1\le y\le S+M,
$$
has at least $T/L$ solutions.
Since there are at most two solutions to the above congruence
with the same value of $x$, by the pigeonhole principle,
there exists an interval of length $4(h+1)M/T$  containing at least
$2(h+1)$ solutions $(x,y)$ with pairwise distinct values $x$. Let
$x_0$ be the first of these values and $(x_0,y_0)$ the solution. It is
clear that $T/L$ is bounded by the number of solutions of
\begin{eqnarray*}
(x_0+x)^{h}\equiv \lambda (y_0+y)^2\pmod p,\\
-M/L\le x\le M/L,\quad -M\le y\le M,
\end{eqnarray*}
which is equivalent to
\begin{equation}
\begin{split}
\label{eq:cong c}
c_{h}x^{h}+\ldots +c_1x+c_0y&\equiv y^2\pmod p,\\
-M/L\le x\le M/L,\quad -&M\le y\le M,
\end{split}
\end{equation}
where
$$
c_0=-2y_0\mand c_j=\lambda^*\binom{h}jx_0^{h-j},\
j=1,\ldots ,h,
$$
where $\lambda^*$ is defined by $\lambda^* \lambda \equiv 1 \pmod p$
and $1\le \lambda^* < p$. In particular, $c_h\not\equiv 0\pmod p$.
Besides, there are at least $2h+1$ solutions $(x,y)$ of~\eqref{eq:cong c}
with $x$ pairwise distinct and such that $1\le x\le 4(h+1)M/T$. From
these $2h+1$ values we fix $h$: $(x_1,y_1),\ldots
,(x_{h},y_{h})$ and rewrite~\eqref{eq:cong c} in the form
\begin{equation}\label{eq:cong 2}
\left (\begin{matrix}x^{h} &\ldots & x & y\\ x_{h}^{h}
&\ldots &
x_{h} & y_{h}\\ & \ldots &\\x_1^{h} &\ldots & x_1 & y_1\\
\end{matrix}\right )\(\begin{matrix}c_{h}\\
\ldots \\ c_1 \\c_0  \end{matrix}\)\equiv
\(\begin{matrix}y^2\\y_{h}^2\\\ldots \\y_1^2
\end{matrix}\)\pmod p.
\end{equation}
Since $h$ is odd, by Lemma~\ref{lem:deter},
we know that at most $2h$
pairs $(x,y)$, with $x$ pairwise distinct, satisfy both the
congruence~\eqref{eq:cong 2} and the congruence
$$
\left |\begin{matrix} x^{h} &\ldots & x & y\\ x_{h}^{h}
&\ldots & x_{h} & y_{h}\\  & \ldots & \\x_1^{h} &\ldots & x_1 &
y_1
\end{matrix}\right |\equiv 0\pmod p.
$$
Since there are at least $2h+1$ solutions of~\eqref{eq:cong 2}, for one of
them, say $(x_{h+1},y_{h+1})$, we have
$$
\Delta=\left |\begin{matrix} x_{h+1}^{h} &\ldots & x_{h+1} & y_{h+1}\\
x_{h}^{h} &\ldots & x_{h} & y_{h}\\  & \ldots &
\\x_1^{h} &\ldots & x_1 & y_1
\end{matrix}\right |\not \equiv 0\pmod p.
$$

Note that $1\le |\Delta |\ll (M/T)^{h(h+1)/2}M$. Now we solve the
system
\begin{equation}\label{eq:cong 3}
\(\begin{matrix}
x_{h+1}^{h} &\ldots & x_{h+1} & y_{h+1}\\
x_{h}^{h} &\ldots &
x_{h} & y_{h}\\ & \ldots &\\x_1^{h} &\ldots & x_1 & y_1\\
\end{matrix}\)
\(\begin{matrix}c_{h}\\ c_{h-1}\\ \ldots \\ c_0  \end{matrix}\)
\equiv
\(\begin{matrix}y_{h+1}^2\\y_{h}^2\\\ldots \\y_1^2\end{matrix}\)
\pmod p
\end{equation}
with respect to $(c_{h},\ldots ,c_1,c_0)$. We write $\Delta_j$ for
the determinant of the matrix on the left hand side where we have
substituted the column $j$ by the vector $(y_{h+1}^2,\ldots
,y_1^2)$. With this notation we have that
$$
c_j=\frac{\Delta_{h+1-j}}{\Delta},\quad j=0,\ldots h,
$$
and the congruence~\eqref{eq:cong c} is equivalent to
$$
\Delta_1x^{h}+\Delta_2x^{h-1}+\ldots
+\Delta_{h}x+\Delta_{h+1}y- \Delta y^2\equiv 0\pmod p.
$$
In particular, $\Delta_1\not\equiv 0\pmod p$. We can
write this congruence as an equation over $\Z$:
\begin{equation}\label{eq:integer}
\Delta_1x^{h}+\Delta_2x^{h-1}+\ldots
+\Delta_{h}x+\Delta_{h+1}y- \Delta y^2=pz,
\qquad z\in \Z.
\end{equation}

We  can easily check that
$$
|\Delta_{h+1}|\ll (M/T)^{h(h+1)/2}M^2
$$
and
$$
|\Delta_j|\ll (M/T)^{h(h-1)/2+j -1}M^3,\qquad j=1,\ldots ,h.
$$
Thus, collecting the above estimates, we derive
\begin{eqnarray*}
|z|&\ll &\frac{1}{p} \(\sum_{j=1}^h | \Delta_j|(M/L)^{h-j+1}+
|\Delta_{h+1}|M+|\Delta|M^2\)\\
&\ll &\frac{M^3}{p} \(\sum_{j=1}^h (M/T)^{h(h-1)/2+j -1}(M/L)^{h-j+1}+
(M/T)^{h(h+1)/2}\)\\
&\ll &\frac{M^3}{p} \(M^{h(h+1)/2}T^{-h(h-1)/2}L^{-h} \sum_{j=1}^h (TL)^{-j +1}+
(M/T)^{h(h+1)/2}\)\\
&\ll
&\frac{M^{h(h+1)/2+3}}{pT^{h(h-1)/2}L^{h}}.
\end{eqnarray*}

Since $h$ is odd, and $\Delta\not=0,\, \Delta_1\not=0$, we have that, for each $z$, the curve~\eqref{eq:integer} is absolutely irreducible.
Thus by Lemma~\ref{lem:BombPila}
it contains at most $M^{1/h+o(1)}$ integer points $(x,y)$
with $|x|,|y|\le M$. Hence
\begin{equation}
\label{eq:prelim}
T\le LM^{1/h+o(1)}
\(1+\frac{M^{h(h+1)/2+3}}{pT^{h(h-1)/2}L^{h}}\)
\end{equation}
for any $L$ satisfying~\eqref{eq:cong L}.

We can assume that the following lower bounds hold for $T$:
\begin{equation}
\label{eq:T large}
T>M^{1/h} \mand T>16(h+1)\(M(M^4/p)^{2/h(h+1)} + 1\)
\end{equation}
since otherwise there is
nothing to prove.

Take $L=\fl{1+(M^{h(h+1)/2+3}/p)^{2/h(h+1)}}$.
We note that~\eqref{eq:cong L} holds as otherwise $L \ge 2$ and
we have
\begin{equation*}
\begin{split}
 \(\frac{M^{h(h+1)/2+3}}p\)^{2/h(h+1)} \ge  L-1 \ge
 \frac L2&>\frac T{16(h+1)}\\
 > M\(\frac{M^4}p\)^{2/h(h+1)}&=
\(\frac{M^{h(h+1)/2+4}}p\)^{2/h(h+1)},
\end{split}
\end{equation*}
which is impossible.

If $M<p^{1/(h(h+1)/2+3)}$ we have $L=1$ and also
$$
\frac{M^{h(h+1)/2+3}}{pT^{h(h-1)/2}L^{h}}   \le
\frac{M^{h(h+1)/2+3}}{p}< 1.
$$
In this case, the bound~\eqref{eq:prelim} yields
$$
T\ll M^{1/h+o(1)}.
$$
If $M\ge p^{1/(h(h+1)/2+3)}$, we have
$$(M^{h(h+1)/2+3}/p)^{2/h(h+1)} \ll L\ll
(M^{h(h+1)/2+3}/p)^{2/h(h+1)}
$$
and, recalling our assumption~\eqref{eq:T large} and the choice
of $L$, we obtain
\begin{equation*}
\begin{split}
&\frac{M^{h(h+1)/2+3}}{pT^{h(h-1)/2}L^{h}} \\
 &\qquad \quad \ll
\frac{M^{h(h+1)/2+3}}{pM^{h(h-1)/2}(M^4/p)^{(h-1)/(h+1)}(M^{h(h+1)/2+3}/p)^{2/(h+1)}}
= 1.
\end{split}
\end{equation*}
Hence, in this case we derive from~\eqref{eq:prelim} that
\begin{equation*}
\begin{split}
T\le (M^{h(h+1)/2+3}/p)^{2/h(h+1)} & M^{1/h+o(1)} \\
\le & M\left(M^4/p\right )^{2/h(h+1)+o(1)},
\end{split}
\end{equation*}
which concludes the proof.

\subsection{Proof of Theorem~\ref{thm:main3}}

Clearly
\begin{equation}
\label{eq:fH moments} \sum_{H \in \cH\(\fB\)} N(H;\fB) = M^{2g}
\mand \sum_{H \in \cH\(\fB\)} N(H;\fB)^2= T(\fB).
\end{equation}
As in~\cite{CSZ}, using~\eqref{eq:fH moments} and  the Cauchy
inequality we derive
$$
\# \cH\(\fB\) \ge M^{4g} T(\fB)^{-1}.
$$

From~\eqref{eq:isom} we observe that $T(\fB)$ is the numbers of
pairs of vectors $(\va,\vb)$, $\va,\vb\in \fB$,  such that there exists
$\alpha$ such that
$$a_i\equiv \alpha^{4g+2-2i}b_i \pmod p, \qquad i=0,\ldots, 2g-1.
$$
In particular,
$$
a_{2g-1}^3b_{2g-2}^2 \equiv  a_{2g-2}^2b_{2g-1}^3 \pmod p.
$$
Now, by~\cite[Theorem~7]{CSZ} we see that there are only
$O\(M^4/p + M^{2+o(1)}\)$ possibilities for the quadruple
$(a_{2g-1},a_{2g-2},b_{2g-1},b_{2g-2})$. When it is fixed, the
parameter $\alpha$ in~\eqref{eq:isom} can take at most 4 values, and
thus for every choice of $(a_0, \ldots, a_{2g-3})$ there are only 4
choices for  $(b_0, \ldots, b_{2g-3})$. Therefore,
\begin{equation}
\label{eq:TB} T(\fB) \le M^{2g-2}\(M^4/p + M^{2+o(1)}\).
\end{equation}

When $M<p^{1/(2g)}$ we obtain $T(\fB)\le M^{2g+o(1)}$ and $\#
\cH\(\fB\)\ge M^{2g+o(1)}$, which proves Theorem~\ref{thm:main3} in
this range.

When $M\ge p^{1/(2g)}$ we use a different approach. Using the
notation
$$N_i(\lambda)=
\# \{(a_i,b_i)~:~a_i/b_i\equiv\lambda \pmod p,\ R_i+1\le
a_i,b_i\le R_i+M\},
$$ we can write
$$
T(\fB)=\sum_{\alpha=1}^{p-1}N_0(\alpha^{4g+2})N_1(\alpha^{4g})\ldots
N_{2g-1}(\alpha^4).
$$
Thus,
\begin{equation*}
\begin{split}
T^{2g}(\fB)&\le \left (\sum_{\alpha=1}^{p-1}N_0^{2g}(\alpha^{4g+2})\right )\ldots \left (\sum_{\alpha\ne
0}N_{2g-1}^{2g}(\alpha^{4})\right )\\ &\le\left
((4g+2)\sum_{\alpha=1}^{p-1}N_0^{2g}(\alpha)\right )\ldots \left
(4\sum_{\alpha=1}^{p-1}N_{2g-1}^{2g}(\alpha)\right )
\end{split}
\end{equation*}
and then we have
$$T(\fB)\ll \max_i\sum_{\alpha=1}^{p-1}N_i^{2g}(\alpha).$$
We observe that for any $\alpha\not \equiv 0 \pmod p$ there exist
integers  $r,s$ with
$1\le |r|,s\le p^{1/2}$, $(r,s)=1$ and such that
$\alpha\equiv r/s \pmod p$. Thus
$$
\sum_{\alpha=1}^{p-1}N_i^{2g}(\alpha)\le \sum_{\substack{1\le
r,s<p^{1/2}\\ \gcd (r,s)=1}}N_i^{2g}(r/s)+\sum_{\substack{1\le
r,s<p^{1/2}\\ \gcd(r,s)=1}}N_i^{2g}(-r/s).
$$

Our estimate of  $N_i(r/s)$ is based on an argument that is very close
to that used in the proof of~\cite[Lemma~1]{ACZ}.
Namely, we observe that $N_i(r/s)$ is the number of solutions $(x,y)$ to the
congruence
$$x/y\equiv r/s\pmod p, \qquad R_i+1\le x,y\le R_i+M,
$$
which is equivalent to the congruence
$$sx-ry\equiv c\pmod p,\quad
1\le x,y\le M,
$$
for a suitable $c$. We can write the congruence as
an equation in integers
$$sx-ry=c+zp,\quad
1\le x,y\le M,\quad z\in \Z.  $$
We observe that
$$
|z|\le
\frac{|s|M+|r|M+|c|}p\le \frac{(|s|+|r|)M}p+1.
$$

For each $z$ we consider, in case it has, a solution $(x_z,y_z)$,
$1\le x_z,y_z\le M$. The solutions of the diophantine equation above
is given by $(x,y)=(x_z+rt,y_z+st)$, $t\in \Z$. The restriction $1\le
x,y\le M$ implies that $|t|\le M/\max\{r,s\}$.

Thus we have
\begin{equation*}
\begin{split}
N_i(r/s)&\le \left (1+\frac{2M}{\max\{r,s\}}\right )\left
(1+\frac{2M(s+r)}p\right )\\
&\le 1+\frac{4M\max\{r,s\}}p+\frac{2M}{\max\{r,s\}}+\frac{4M^2}p.
\end{split}
\end{equation*}
Therefore
\begin{equation*}
\begin{split}
 \sum_{\substack{1\le
r,s<p^{1/2}\\ \gcd (r,s)=1}}&N_i^{2g}(r/s)\\
\ll &\sum_{1\le
r,s<p^{1/2}}\(1+\frac{M^{2g}\(\max\{r,s\}\)^{2g}}{p^{2g}}+
\frac{M^{2g}}{\(\max\{r,s\}\)^{2g}}+\frac{M^{4g}}{p^{2g}}\)\\
\ll &\sum_{1\le r<s<p^{1/2}}\left
(1+\frac{M^{2g}s^{2g}}{p^{2g}}+\frac{M^{2g}}{s^{2g}}+
\frac{M^{4g}}{p^{2g}}\right )\\
\ll &\sum_{1\le s<p^{1/2}}\left
(s+\frac{M^{2g}s^{2g+1}}{p^{2g}}+
\frac{M^{2g}}{s^{2g-1}}+\frac{M^{4g}s}{p^{2g}}\right)\\
\ll &p+\frac{M^{2g}}{p^{g-1}}+M^{2g}\sum_{1\le s<p^{1/2}}\frac
1{s^{2g-1}}+\frac{M^{4g}}{p^{2g-1}}.
\end{split}
\end{equation*}
The estimate of the sum with $N_i^{2g}(-r/s)$ is fully analogous.

Assume that $M\ge p^{1/(2g)}$ and observe that
$$
\sum_{1\le s<p^{1/2}}\frac 1{s^{2g-1}}
\ll \begin{cases} \log M , & \text{ if } g=1,\\
1, & \text{ if } g\ge 2.\end{cases}
$$
Thus we have
\begin{equation}\label{T(B)}
T(\fB)\ll \begin{cases}M^2\log
M+M^4/p,& \text{ if } g=1,\\ M^{2g}+M^{4g}/p^{2g-1},& \text{ if } g\ge
2,\end{cases}\end{equation}
which gives
$$
\# \cH\(\fB\) \ge M^{4g} T(\fB)^{-1}\gg \begin{cases}\min\{p,M^{2+o(1)}\}, & \text{ if } g=1,\\
\min\{p^{2g-1},M^{2g}\},& \text{ if } g\ge 2,\end{cases}
$$
and proves
Theorem~\ref{thm:main3} in the range $M\ge p^{1/2g}$.

\section{Comments}

The problem of obtaining a nontrivial upper bound for $I_{f}(M;R,S)$
in the range $p^{1/3}<M<p^{1/2}$ is still  open.

On the other hand, we note that using bounds of exponential sums
obtained
with the method of Vinogradov instead of Lemma~\ref{lem:Weyl},
see~\cite{BoWo,Ford,Par,Vaughan} and references therein,
also leads  to some nontrivial on $J_f(M;R,S)$
but these results seem to be weaker than
a combination
of Theorem~\ref{thm:Polynomial Map}  with the bounds from~\cite{CGOS}.

Similar ideas can be exploited to obtain lower bounds for the cardinality
of the set
$\cI(\cB)$
of non-isomorphic isogenous elliptic curves $\Ha$ with coefficients in
a cube $\cB$.

Indeed, let us denote by $\cI_t$ the isogeny class consisting of
elliptic curves over $\F_p$  with the same number $p+1-t$ of $\F_p$-rational points.
By a result of Deuring~\cite{Deu},
each admissible value of $t$, that is, with $|t|\leq 2p^{1/2}$, is taken and hence
there are about $4p^{1/2}$ isogeny classes. Furthermore,
Birch~\cite{Bir} has actually given a formula via the Kronecker class
number for the number of isomorphism classes of elliptic curves over a
finite field $\F_q$ lying in $\cI_t$. Finally, Lenstra~\cite{Lens} has
 obtained upper
and lower bounds for this number and, in particular,  shown that
the number of isomorphism classes of elliptic curves of a given order
is $O\(p^{1/2}\log p \(\log\log p\)^{2}\)$.

Observe that once again bounds for $N(H;\fB)$ can be translated into
bounds for the number of isogenous non-isomorphic curves with
coefficients in $\fB$, via multiplication by $p^{1/2+o(1)}$. However,
as we have done before, one can obtain better bounds in terms of $T(\fB)$
which is given by~\eqref{eq:fH moments}.

Thus, using~\eqref{eq:fH moments} and~\eqref{T(B)}, with $g=1$,
we see that for the set  $\cH(t,\fB)$ of elliptic curves   $\Ha \in
\cI_t$ with $\va \in \fB$, we have
\begin{align*}
\# \cH(t,\fB) &= \sum_{H\in \cH\(\fB\)\cap \cI_t }
 N(H,\fB)\\ &\leq
(\#\cI_t)^{1/2}\(\sum_{H\in\cH\(\fB\)}N(H,\fB)^2\)^{1/2}=(\#\cI_t)^{1/2} T(\fB)^{1/2}\\
&\ll \(M^2p^{-1/4}+p^{1/4}M\log^{1/2}M\)(\log p)^{1/2}\log\log p.
\end{align*}
This improves the trivial bound
$$\cH(N,\fB) \ll \min\{M^2, p^{3/2} (\log p)^{1/2}\log\log p\}
$$
for $p^{1/4+\varepsilon} \le M \le p^{7/8-\varepsilon}$ (with any fixed $\varepsilon>0$). Furthermore, it also implies the lower bound
\begin{align*}
\# \cI(\cB) & \gg \frac{M^2}{\max_{|t|\in 2p^{1/2}} \cH(t,\fB)} \\
& \gg \min\{p^{1/4}, Mp^{-1/4}\log^{-1/2}M\}
(\log p)^{-1/2}(\log\log p)^{-1}.
\end{align*}

\section*{Acknowledgements}

The authors are grateful to Alfred Menezes for discussions
and useful references on isomorphism classes of hyperelliptic
curves.

M.-C. Chang  is very grateful to the Department of Mathematics
of the
University of California at Berkeley for its hospitality.

During the preparation
of this paper, M.-C.~Chang was supported in part by NSF, J.~Cilleruelo was supported by
Grant MTM 2008-03880 of MICINN (Spain), M.~Z.~Garaev was  supported in part by the Red Iberoamericana de Teor\'ia de N\'umeros, I.~E.~Shparlinski was supported in part
by ARC grant DP1092835
and  by NRF Grant~CRP2-2007-03, Singapore. A.~Zumalac\'arregui was supported by a FPU grant from Ministerio de Educaci\'on, Ciencia y
Deporte, Spain.

\end{document}